\documentclass[12pt]{amsart}
\usepackage{amssymb, verbatim,url}
\usepackage{enumerate,mathtools}
\usepackage{color}
\usepackage{leftidx}

\newcommand{\R}{{\mathbb R}}

\newcommand{\D}{\mathcal{D}}

\newcommand\norm[1]{\left\lVert#1\right\rVert}

\newcommand{\erf}{\mathrm{erf}\,}

\DeclarePairedDelimiter\abs{\lvert}{\rvert}

\date{\today}

\newtheorem{theorem}{Theorem}[section]

\newtheorem{lemma}[theorem]{Lemma}

\theoremstyle{definition}
\newtheorem{Definition}[theorem]{Definition}

\theoremstyle{remark}

\numberwithin{equation}{section}

\usepackage[T1]{fontenc}

\begin{document}

\title[More generalized groundwater model]
{More generalized groundwater model with space-time Caputo Fabrizio fractional differentiation}

\author[Djida]{Jean-Daniel Djida}
\author[Atangana]{Abdon Atangana}

\address[Djida]{African Institute for Mathematical Sciences (AIMS), P.O. Box 608, Limbe Crystal Gardens, South West Region, Cameroon. }
\email[Djida]{jeandaniel.djida@aims-cameroon.org}

\address[Atangana]{Institute for Groundwater Studies, Faculty of Natural and Agricultural Sciences, University of the Free State, 9300, Bloemfontein, South Africa.}
\email[Atangana]{abdonatangana@yahoo.fr}


\begin{abstract}
We prove existence and uniqueness of the flow of water within a confined aquifer with fractional diffusion in space and fractional time derivative in the sense of Caputo-Fabrizio. Our main method is the fixed-point theorem. We propose the numerical approximation of the model. The Crank-Nicolson numerical scheme was used to solve the modified model. In order to check the effectiveness of the model, stability analysis of the numerical scheme for the new model are presented.  
\end{abstract}

\maketitle

\section{Introduction}
The flow of subsurface water within the geological formation is an interested physical problem that has attracted attention of many scholars around the world due to its complexities. The first model was suggested by Theis, where he adapted the model of heat flow in an homogeneous media. Although his equation has been extensively used in the circle of geohydrology but the comparison of mathematical equation with experimental data show a disagreement. However although the collection of experimental data could be bias but still the nature revealed the true behind a physical problem. It is then believed that the mathematical equation suggested by their needs to be reverified. To do this, the time derivation in the Theis equation was replaced several time by the concept of fraction differentiation see \cite{Cloot}. Nevetheless the kernel used in the concept of fractional differentiation is the well-known power law $x^{-\alpha}$ which has a singularity as $x \to 0$. However change in water level is observed in the vicinity of the borehole which is  consider here as origin. But at the origin with the power we have a singularity.To improve this, the kernel was recently replaced by the exponential decay law which does not have singularity event at the origin~\label{Fabrizio}.\medskip

The aim of this work is to revisit the groundwater flow model with the Caputo-Fabrizio derivative. For this purpose, the problem of flow of water within a confined aquifer.\medskip

Very recently a new model of groundwater flow in a confined aquifer has been proposed by A. Atangana and Alkahtani \cite{Atangana2015} in order to take into account the effect of different scales in the annular space as \medskip
\begin{equation}
F_{2} - F_{1} = {}^{\textsc{CF}}\D_t^{\alpha}V(r,t),\quad 0 < \alpha \leq 1, 
\end{equation}
where, $F_{1}$ and $F_{2}$ are inflow and outflow rate, respectively, and $V(r,t)$ is the volume. ${}^{\textsc{CF}}\D_t^{\alpha}$ denotes the Caputo-Fabrizio fractional derivative given by~ \eqref{eq:Caputo-Fabrizio_Derivative}. The variation for the radial flow through into a well for confined homogeneous and isotropic aquifer is given by
\begin{equation}\label{flowvariation}
F_{2} - F_{1} = K\bigg[ \partial_r h(r,t) + \partial^{2}_r h(r,t) \bigg] \big( 2\pi (r + dr)b \big) - K\bigg[ \partial_r h(r,t) \bigg] \big( 2\pi r b \big), 
\end{equation}
where $r$ is the radio of the annular cylinder, $b$ is the thickness of the confined aquifer, and $K$ is the hydraulic conductivity of the aquifer \cite{Atangana2015}.\medskip

The change of volume at different scales is proportional to the change in hydraulic head at different scales in the confined aquifer. This change has been presented by A.Atangana and Alkahtani  \cite{Atangana2015} as 
\[
{}^{\textsc{CF}}\D_t^{\alpha}V(t,r) = 2\pi r dr {}^{\textsc{CF}}\D_t^{\alpha}h(t,r).
\]
By putting all together one can get the equation in the form
\begin{equation}\label{Abdon_model}
\frac{S}{T_{r}} {}^{\textsc{CF}}\D_t^{\alpha}h(t,r)= \frac{1}{r}\partial_r h(r,t) + \partial^{2}_r h(r,t), 
\end{equation}
where $T$ is the transmissivity of the aquifer and $S$ the Storativity.\medskip

The above equation~\eqref{Abdon_model} which describes the flow of water at different scales in time within the confined aquifer has been study analytically and numerically \cite{Atangana2015}. This model which is fractional in time can be fractional in space as well. This could useful to get the global picture of the behaviour of the of water within the confined homogeneous aquifer. As one of our aim is to enhance the mathematical model describing the flow of water in a confined aquifer, we proposed the model given by \eqref{prob}. This model is based on the Caputo-Fabrizio derivative in time, and fractional Laplacian in space.\medskip

This paper is devoted to study the existence, uniqueness and stability of of solution to the problem of ground water flow within a confined aquifer in different scales in time and in space.\medskip

Let $\Omega$ be the domain of the confined aquifer and an open and bounded subset of $\R^{n}~(n \geq 1)$, with boundary $\partial \Omega$. Given $\alpha,\in (0,1)$,  the specific equation of groundwater flow we study is 
\begin{equation}  \label{prob}
\begin{cases}
\partial_t^{\alpha} h = \vartheta(r)\partial _{r}h + \gamma \mathcal{L}^{\alpha} h, \\
h_{0} = h(r,0) \quad \text{on} \quad \partial \Omega \times (0,T),
\end{cases}
\end{equation}
where the initial datum $h_{0} = h(r,0)$  and we seek $h(r,t)$ to be the head. $\vartheta(r) = \frac{T_{r}}{rS}$ is given in meters per second. Physical considerations restrict $\alpha \in (0,1)$. Notice that, $h$ is a function of $r$ and the time $t > 0$, and the coefficient $\gamma = T_{r}S^{-1}$ is given in meters $\alpha$ per second. For the case $\alpha = 1$ then $\gamma$ is given in meters$^{2}$ per second.\medskip

The problem we consider uses the derivative with fractional order in time \eqref{eq:Caputo-Fabrizio_Derivative} as well as the fractional Laplacian in the sense of Caputo-Fabrizio \eqref{Caputo_Fractional_Laplacian}.

$\mathcal{L}^{\alpha}$ denotes the Caputo-Fabrizio fractional space derivative \eqref{Caputo-Fabrizio_Derivative} of order $\alpha \in (0,1)$. The Equation \eqref{prob} is then a parabolic equation which is nonlocal.\medskip

Notice that for $\alpha=1$, the parabolic problem given by the Equation~\eqref{prob} becomes a classical groundwater flow problem. The pressure of water is related to the density via a nonlocal operator. In our case we consider a pressure which takes into account long range interactions effects. The space derivative which involve the exponential decay--Caputo-Fabrizio type has slimly the form of the fading kernel.\medskip

This paper is organized as follows: In Section \ref{Sec:2}, we recall some properties of Caputo-Fabrizio fractional derivatives and Laplacian. Existence and uniqueness for the problem \eqref{prob} are discussed in Section \ref{Sec:3}. Furthermore, in Section \ref{Sec:4}, numerical analysis of the new groundwater model is presented. Finally \ref{Sec:5} is dedicated to our perspectives and conclusions.

\section{Properties of Caputo-Fabrizio fractional derivatives and Laplacian}\label{Sec:2}

In the following we recall the definitions of fractional derivative and integral in the sense of Caputo-Fabrizio \cite{Fabrizio, Losada} that will be useful.
\begin{Definition}\label{Caputo-Fabrizio_Derivative}
Let $h \in (0,T) \times H^{1}_{0}(\Omega)$, $\alpha \in (0,1)$. The Caputo-Fabrizio derivative of order $\alpha$ of a function $h$ is defined by
\begin{equation}\label{eq:Caputo-Fabrizio_Derivative}
^{\textsc{CF}}\D_t^{\alpha}h(r,t)= \dfrac{B(\alpha)}{1-\alpha}\int_0^t \partial_{s}h(r,s) \exp\bigg[-\frac{\alpha}{1-\alpha}(t-s)\bigg] ds, \qquad t > 0,
\end{equation}
where 
\[
B(\alpha) = 1 - \alpha + \frac{\alpha}{\Gamma(\alpha)},
\]
is a normalization constant depending on $\alpha$ such that $B(0) = B(1) = 1$. 
\end{Definition}
According to the fractional derivative in the Caputo-Fabrizio sense \eqref{eq:Caputo-Fabrizio_Derivative}, contrary to the old definition \cite{Ivan2015}, the new kernel has no singularity for $t = s$, but still need more regularity. Despite some weakness of this derivative, it has been shown recently by many researchers that this new derivative can be used with great success for those problems described by Caputo and Fabrizio \cite{Abdon2}. \medskip

The fractional integral associated to the  Caputo-Fabrizio derivative \eqref{Caputo_Fabrizio_Integral} is given as
\begin{equation} \label{Caputo_Fabrizio_Integral}
I_{t}^{\alpha}h(r,t) = \frac{2(1-\alpha)}{(2-\alpha) B(\alpha)}h(r,t) + \frac{2 \alpha}{(2-\alpha) B(\alpha)} \int_{0}^{t} h(r,s) ds, ~~ t > 0.
\end{equation}

\begin{Definition} \label{Def_Caputo_Fractional_Laplacian}
Let $h \in (0,T) \times H^{1}_{0}(\Omega)$. The Fractional Laplacian in the sense of Caputo derivative for $\alpha \in (0,1)$ is defined as
\begin{equation}\label{Caputo_Fractional_Laplacian}
\mathcal{L}^{\alpha} h(r,t) = \frac{\alpha}{(1-\alpha)\sqrt{\pi}} \int_{0}^{r} \exp \big[ -\alpha^{2} (r - \tau)^{2} \big] \partial_{\tau\tau}h(\tau,t) d\tau,
\end{equation}
where $\mathcal{L}^{\alpha}$ stand for 
$ (\nabla^{2})^{\alpha} $  and $\lambda = - \alpha(1-\alpha)^{-1}$.
\end{Definition}

\section{Existence and uniqueness for the new model of groundwater flow in confined aquifer}\label{Sec:3}
\subsection{Formulation of the problem and existence of solutions}
Integrating the Equation \eqref{prob}, using the associate fractional integral \eqref{Caputo_Fabrizio_Integral}, yields to
\begin{align}\label{exact_solution}
h(r,t) - h(r,0) & = I_{t}^{\alpha}\big[ \vartheta(r)\partial_{r}h(r,t) + \gamma \mathcal{L}^{\alpha} h(r,t) \big]  \nonumber \\
& = \frac{2(1-\alpha)}{(2-\alpha) B(\alpha)}\mathcal{K}(r,t,h) + \frac{2 \alpha}{(2-\alpha) B(\alpha)} \int_{0}^{t} \mathcal{K}(r,s,h) ds,
\end{align}
where the kernel $\mathcal{K}(r,t,h)$ is defined as
\begin{align*}
\mathcal{K}(r,t,h) & = \vartheta(r)\partial_{r}h(r,t) + \gamma  \frac{\alpha}{(1-\alpha)\sqrt{\pi}} \int_{0}^{r} \exp \big[ -\lambda^{2} (r - \tau)^{2} \big] \partial_{\tau\tau}h(\tau,t) d\tau \\
& = \vartheta(r)\partial_{r}h(r,t) + \eta \partial_{r}h(r,t) + \eta \int_{0}^{r} \exp \big[ -\lambda^{2} (r - \tau)^{2} \big]h(\tau,t)d\tau \\
& - 2\lambda^{2} \eta \int_{0}^{r} (r - \tau)^{2} \exp \big[ -\lambda^{2} (r - \tau)^{2} \big]h(\tau,t)d\tau 
\end{align*}
where $\eta =   \gamma \alpha \big[(1-\alpha)\sqrt{\pi} \big]^{-1} $. \medskip

Now let us show that the nonlinear kernel $\mathcal{K}(r,t,h)$ satisfies the Lipschitz condition.

\begin{theorem}
Let $\alpha \in (0,1)$,~$T > 0$ and $\mathcal{K}: (0,T) \times H^{1}_{0}(\Omega)$ a continuous function such that there exists $\Lambda > 0$ satisfying,
\[
\norm{ \mathcal{K}(r,t,h) - \mathcal{K}(r,t,\varphi) } \leq \Lambda \norm{h - \varphi }, \qquad \text{for all}~~h, \varphi \in (0,T) \times H^{1}_{0}(\Omega).
\]
If $(\varepsilon_{1} + \varepsilon_{2} + \varepsilon_{3}) < 1$, the operator $\mathcal{K}$ is a contraction.
\end{theorem}
\begin{proof}
We consider two bounded functions $h$ and $\varphi$ in $H^{1}_{0}(\Omega)$. We have by triangular inequality
\begin{align*}
& \norm{ \mathcal{K}(r,t,h) - \mathcal{K}(r,t,\varphi) } \leq \vert \vartheta(r) \vert \norm{\partial_{r} \big[ h(r,t)-  \varphi(r,t) \big] } + \vert \eta \vert \norm{ \partial_{r} \big [h(r,t) - \varphi(r,t) \big] }\\
& + \vert 2\lambda^{2} \eta \vert \norm{ \int_{0}^{r} (r - \tau)^{2} \exp \big[ -\lambda^{2} (r - \tau)^{2} \big] \big[ h(\tau,t) - \varphi(\tau,t)\big]d\tau} \\
& + \vert \eta \vert \norm{ \int_{0}^{r} \exp \big[ -\lambda^{2} (r - \tau)^{2} \big]\big[ h(\tau,t) - \varphi(\tau,t)\big]d\tau }
\end{align*}

\begin{align*}
& \norm{ \mathcal{K}(r,t,h) - \mathcal{K}(r,t,\varphi) }  \leq \vert \vartheta(r)+ \eta \vert \norm{h(r,t)-  \varphi(r,t)} \\
& + \vert \eta \vert \norm{ \int_{0}^{r} \exp \big[ -\lambda^{2} (r - \tau)^{2} \big]\big[ h(\tau,t) - \varphi(\tau,t)\big]d\tau }\\
& + \vert 2\lambda^{2} \eta \vert \norm{ \int_{0}^{r} (r - \tau)^{2} \exp \big[ -\lambda^{2} (r - \tau)^{2} \big] \big[ h(\tau,t) - \varphi(\tau,t)\big]d\tau} 
\end{align*}
Now by applying Cauchy-Schwartz inequality we have
\begin{align*}
& \norm{ \mathcal{K}(r,t,h) - \mathcal{K}(r,t,\varphi) } \leq  \varepsilon_{1} \norm{h(r,t)- \varphi(r,t)} \\
& + \vert 2\lambda^{2} \eta \vert \bigg( \int_{0}^{r}\norm{ (r - \tau)^{2} \exp \big[ -\lambda^{2} (r - \tau)^{2} \big]}^{2}d\tau \bigg)^{\frac{1}{2}} \bigg( \int_{0}^{r} \norm{ h(\tau,t) - \varphi(\tau,t)}d\tau \bigg)^{\frac{1}{2}} \\
& + \vert \eta \vert  \bigg( \int_{0}^{r} \norm{ \exp \big[ -\lambda^{2} (r - \tau)^{2} \big] }^{2} d\tau \bigg)^{\frac{1}{2}} \bigg( \int_{0}^{r} \norm{ h(\tau,t) - \varphi(\tau,t)}^{2}d\tau \bigg)^{\frac{1}{2}}
\end{align*}
Thus we obtain
\[
\norm{ \mathcal{K}(r,t,h) - \mathcal{K}(r,t,\varphi) } < \Lambda \norm{h(r,t)- \varphi(r,t)},
\]
where $ \Lambda < 1$ and can be estimate as
\[
\Lambda = \varepsilon_{1} + \varepsilon_{2} + \varepsilon_{3} = \frac{S}{T} + \eta  + \gamma \alpha\big[(1-\alpha)\sqrt{\pi} \big]^{-1} + \frac{3\sqrt{2}}{16\lambda^{5}}\eta 
\]

We conclude that the operator $\mathcal{K}$ is a contraction. The statement follows now from Banach's Fixed Point Theorem.
\end{proof}

In the following we show that the solution of our problem \eqref{prob}, given by the Equation \eqref{exact_solution} can be written as an iteration for a given subsequence $h_{m}(r,t) \in (0,T) \times H^{1}_{0}(\Omega)$. \medskip

\begin{theorem}
Assume that a bounded sequence $h_{m}(r,t)$ on $(0,T) \times H^{1}_{0}(\Omega)$ converge to the exact solution of the problem \eqref{prob}, then any bounded sub-sequences on $(0,T) \times H^{1}_{0}(\Omega)$ converge to the exact solution and is a Cauchy sequence with respect to the norm in $H^{1}_{0}(\Omega)$.
\end{theorem}

\begin{proof}
The kernel $\mathcal{K}$ being then bounded on $(0,T) \times H^{1}_{0}(\Omega)$, there exists a subsequence $\mathcal{K}(r,t,h_{m})$ on $(0,T) \times H^{1}_{0}(\Omega)$ that converges on $(0,T) \times L^{2}(\Omega)$ by the \textbf{Rellich-Kondrachov theorem}\cite{ciarlet2013}; furthermore, the difference between two consecutive sub-sequences $\mathcal{K}(r,t,h_{m})$ and $\mathcal{K}(r,t,\varphi_{m})$ also converges on $(0,T) \times L^{2}(\Omega)$. Thus the sub-sequence $\mathcal{K}(r,t,h_{m})$ is a Cauchy sequence with respect to the norm in $H^{1}_{0}(\Omega)$.\medskip

One can reformulate the previous statement as:\medskip

Let $h_{m}(r,t)$ and $h_{m-1}(r,t)$ two successive sub-sequences on $(0,T) \times H^{1}_{0}(\Omega)$. From the Equation \eqref{exact_solution} it follows that
\begin{align}\label{recursive_kernel}
h_{m}(r,t) - h_{m-1}(r,t) 
& = \frac{2(1-\alpha)}{(2-\alpha) B(\alpha)}\bigg[\mathcal{K}(r,t,h_{m-1}) - \mathcal{K}(r,t,h_{m-2}) \bigg] \\
& + \frac{2 \alpha}{(2-\alpha) B(\alpha)} \int_{0}^{t} \bigg[\mathcal{K}(r,s,h_{m-1}) - \mathcal{K}(r,s,h_{m-2}) \bigg] ds.
\end{align}
Next we want to control the difference between the two sub-sequences $h_{m}(r,t)$ and  $h_{m-1}(r,t)$. A direct application of the triangular inequality yields
\begin{align*}\label{recur_1}
&\norm{ h_{m}(r,t) - h_{m-1}(r,t) }
\leq \vert \frac{2(1-\alpha)}{(2-\alpha) B(\alpha)}\vert \norm{ \mathcal{K}(r,t,h_{m-1}) - \mathcal{K}(r,t,h_{m-2}) } \\
& + \vert \frac{2 \alpha}{(2-\alpha) B(\alpha)} \vert \norm{ \int_{0}^{t} \bigg[\mathcal{K}(r,s,h_{m-1}) - \mathcal{K}(r,s,h_{m-2}) \bigg]ds}. \\
& \leq  \frac{2(1-\alpha)}{(2-\alpha) B(\alpha)} \norm{ \mathcal{K}(r,t,h_{m-1}) - \mathcal{K}(r,t,h_{m-2}) } \\
& + \frac{2 \alpha}{(2-\alpha) B(\alpha)}  \int_{0}^{t} \norm{ \mathcal{K}(r,s,h_{m-1}) - \mathcal{K}(r,s,h_{m-2})} ds. 
\end{align*}
Now since the nonlinear kernel given by the operator $\mathcal{K}$ in a contraction, its follows that
\begin{align*}
&\frac{2(1-\alpha)}{(2-\alpha) B(\alpha)} \norm{ \mathcal{K}(r,t,h_{m-1}) - \mathcal{K}(r,t,h_{m-2}) } \\
& + \frac{2 \alpha}{(2-\alpha) B(\alpha)}  \int_{0}^{t} \norm{ \mathcal{K}(r,s,h_{m-1}) - \mathcal{K}(r,s,h_{m-2})} ds. \\
& \leq \frac{2(1-\alpha)}{(2-\alpha) B(\alpha)} \Lambda \norm{ \mathcal{K}(r,t,h_{m-1}) - \mathcal{K}(r,t,h_{m-2}) } \\
& + \frac{2 \alpha}{(2-\alpha) B(\alpha)} \Lambda \int_{0}^{t} \norm{ \mathcal{K}(r,s,h_{m-1}) - \mathcal{K}(r,s,h_{m-2})} ds. 
\end{align*}
Hence there exist a solution $h$ to the problem state by equation \eqref{prob}.

\end{proof}

\subsection{Uniqueness of the exact solution of the problem}
In this section we propose to study the uniqueness of the exact solution given by the problem \eqref{prob}. To do this, we assume that there exists another solution of the problem \eqref{prob}, namely $\varphi(r,t)$. \medskip

Before state the theorem of uniqueness of the problem  \eqref{prob}, we recall the following Lemma~\cite{Losada}. 

\begin{lemma}[Lemma 1 \cite{Losada}]\label{lem}
Let $\alpha \in (0,1)$ and $h$ be a solution of the following fractional differential equation,
\begin{equation}
\partial_t^{\alpha} h(r,t) = 0, \qquad t > 0.
\end{equation}
Then, $h$ is a constant function.
\end{lemma}
Next we can now state the theorem of uniqueness of solution of the problem \eqref{prob}.
\begin{theorem}
Let $\alpha \in (0, 1)$. Then, the solution of the problem of groundwater flow given by Equation \eqref{prob} is unique.
\end{theorem}
\begin{proof}
The approach of our proof comes from \cite{Losada}.\medskip

Suppose that the problem \eqref{prob} has two solutions $h(r,t)$ and $\varphi(r,t)$ that can be written in the form of the Equation \eqref{exact_solution}. This means that
\begin{equation}
\partial_t^{\alpha} h(r,t) - \partial_t^{\alpha} \varphi(r,t) = \partial_t^{\alpha} \big[ h(r,t) - \varphi(r,t) \big] = 0, \qquad \text{and} \quad \big[ h(r,0) - \varphi(r,0) \big] = 0.
\end{equation}
Thus it comes out for all $t \in (0,T)$, from Lemma \ref{lem} and the Equation \eqref{exact_solution} that, $h(r,t) = \varphi(r,t) $. \\
One can then conclude that the problem of groundwater flow describe by the Equation \eqref{prob} has an unique solution given by the Equation \eqref{exact_solution}.
\end{proof}

\section{Numerical analysis of the new groundwater model}\label{Sec:4}

In this section, the numerical approximation of the problem \eqref{prob} is analyse. The stability of the problem using the Fourier method is also presented. The problem will be solved using the Crank-Nicholson scheme.\medskip

\subsection{Discretization of the problem using Crank-Nicholson scheme}
We consider some positive integers $M$ and $N$. The grid points in time and space are defined respectively by $t_{k} = k\tau, \quad k = 0, 1,2, \dots, N$ and $r_{j} = \xi j, \quad j = 0, 1,2, \dots, M$. We also denote by $h_{j}^{k} = h(r_{j}, t_{k})$, the values of the functions $h$ at the grid points. \medskip

The first and the second order approximation of the local derivative is the sense of Crank-Nicholson is given as 

\[
\partial_{r}h(r_{j}, t_{k}) = \frac{1}{2} \bigg[ \frac{ \big( h^{k+1}_{j+1}-h^{k+1}_{j-1} \big) + \big( h^{k}_{j+1} - h^{k}_{j-1} \big) }{2\xi} \bigg],
\]
and
\[
\partial_{rr}h(r_{l}, t_{k}) = \frac{1}{2} \bigg \{ \frac{\big(h^{k+1}_{l+1} - 2h^{k+1}_{l} + h^{k+1}_{l-1}\big) + \big(h^{k}_{l+1} - 2 h^{k}_{l} + h^{k}_{l-1}\big)}{2 \xi ^{2}} \bigg \}.
\]
For discrete version of the Caputo-Fabrizio derivative and Laplacian, we recall that this has been already done by Atangana and Nieto in \cite{\cite{Abdon_Nieto}}. Thus the following Theorems.

\begin{theorem}[Atangana and Nieto~ \cite{Abdon_Nieto}]\label{theorem_discrete_Caputo_Fabrizio }
Let $h(t)$ be a function in $H^{1}_{0}(\Omega) \times (0, T)$, then the first order fractional derivative of the Caputo-Fabrizio derivative of order $\alpha \in (0,1)$ at a point $t_{k}$ is
\begin{equation}
\partial_t^{\alpha} \big[ h(t_{k})\big] = \frac{B(\alpha)}{\alpha} \sum_{s = 1}^{k}\bigg( \frac{h^{s+1}_{j}-h^{s}_{j}}{\tau} + \mathcal{O}(\tau)  \bigg)E_{k,s,\tau},
\end{equation}
where
\[
E_{k,s,\tau} = -\exp\big[\lambda \tau\big(k - s + 1\big)\big] + \exp\big[\lambda \tau \big(k -s\big)\big].
\]
\end{theorem}

\begin{theorem}[[Atangana and Nieto~ \cite{Abdon_Nieto}]\label{theorem_discrete_Caputo_Fabrizio_Laplacian}
Let $h(r,t)$ be a function in  $H^{1}_{0}(\Omega) \times (0, T)$, then the fractional Laplacian in the sense of the Caputo-Fabrizio derivative of order $\alpha \in (0,1)$ at the grid points $(r_{j}, t_{k})$ is given as
\begin{gather} 
\mathcal{L}^{\alpha} \big[ h(r_{j}, t_{k})\big]  = \frac{1}{2} \sum_{l = 1}^{j}\bigg \{ \frac{\big(h^{k+1}_{l+1} - 2h^{k+1}_{l} + h^{k+1}_{l-1}\big) + \big(h^{k}_{l+1} - 2 h^{k}_{l+1} + h^{k}_{l-1}\big)}{2\xi ^{2}} \bigg \}  \nonumber \\
\bigg\{-\erf \big[ -\lambda (r_{j} - r_{l+1}) \big] + \erf \big[-\lambda (r_{j} - r_{l}) \big] \bigg \} + \mathcal{O}\big(\xi ^{2}\big)
\end{gather} 
\end{theorem}
The complete proof of Theorems \ref{theorem_discrete_Caputo_Fabrizio } and \ref{theorem_discrete_Caputo_Fabrizio_Laplacian} can be found in reference \cite{Abdon_Nieto}. \medskip

The results from Theorems \ref{theorem_discrete_Caputo_Fabrizio } and \ref{theorem_discrete_Caputo_Fabrizio_Laplacian} can then be applied to \eqref{prob}. Hence the resulting equation of the discretization of the problem \eqref{prob} using the Crank-Nicholson numerical scheme can be written as 
\begin{align}\label{discrete_version}
&\frac{B(\alpha)}{\alpha} \sum_{s = 1}^{k}\bigg( \frac{h^{s+1}_{j}-h^{s}_{j}}{\tau} \bigg) \bigg[ -\exp\big[\lambda \tau\big(k-s+1\big)\big] + \exp\big[\lambda \tau \big(k - s\big)\big]
 \bigg] \nonumber \\
& - \frac{1}{2} \vartheta(r_{j}) \bigg[ \frac{ \big( h^{k+1}_{j+1}-h^{k+1}_{j-1} \big) + \big( h^{k}_{j+1} - h^{k}_{j-1} \big) }{2 \xi} \bigg] \\
&-\frac{1}{2} \sum_{l = 1}^{j}\bigg \{ \frac{\big(h^{k+1}_{l+1} - 2h^{k+1}_{l} + h^{k+1}_{l-1}\big) + \big(h^{k}_{l+1} - 2 h^{k}_{l+1} + h^{k}_{l-1}\big)}{2 \xi^{2}} \bigg \}  \nonumber \\
&\bigg\{-\erf \big[-\lambda (r_{j} - r_{l+1}) \big] + \erf \big[-\lambda (r_{j} - r_{l} ) \big] \bigg \} = 0.\nonumber 
\end{align}
For the simplicity of notation, we let
\begin{align}
E_{k,s,\tau} &= \frac{B(\alpha)}{\alpha \tau} \bigg[-\exp\big[\lambda \tau\big(k-s+1\big)\big] + \exp\big[-\lambda \tau \big(k -s\big)\big] \bigg],\nonumber \\
F_{k,j,l} &=  \frac{1}{2\xi^{2}}\bigg[-\erf \big[-\lambda (r_{j} - r_{l+1}) \big] + \erf \big[-\lambda (r_{j} - r_{l} ) \big] \bigg], \nonumber\\
G_{\alpha} & = \frac{1}{2\xi}.\nonumber
\end{align}
Hence the Equation\eqref{discrete_version} becomes 
\begin{align}\label{discrete_full_version}
&\sum_{s = 1}^{k}\big(h^{s+1}_{j}-h^{s}_{j} \big) E_{k,s,\tau} - \frac{1}{2}G_{\alpha} \vartheta(r_{j}) \bigg[ \big( h^{k+1}_{j+1}-h^{k+1}_{j-1} \big) + \big(  h^{k}_{j+1} - h^{k}_{j-1} \big) \bigg] \nonumber  \\
&-\frac{1}{2} \sum_{l = 1}^{j}\bigg[ \big(h^{k+1}_{l+1} - 2h^{k+1}_{l} + h^{k+1}_{l-1}\big) + \big(h^{k}_{l+1} - 2 h^{k}_{l} + h^{k}_{l-1}\big) \bigg] F_{k,j,l} = 0,
\end{align}
with initial and boundary conditions 
\begin{equation*}
\begin{cases}
h^{0}_{j} = \phi(r_{j}),\qquad \qquad (1 \leq j \leq M ), \\ h^{k}_{0} = 0, ~h^{k}_{M} = 0, \qquad (0 \leq k \leq N ).
\end{cases}
\end{equation*}
\subsection{Stability analysis}
We analyse the stability of our discrete problem \eqref{discrete_full_version} using the Fourier method as presented in \cite{Abdon_Nieto, Karatay}. \medskip

Let $\tilde{h}^{k}_{j}$ be the approximate solution of our problem and we define by $\delta_{j}^{k} = h^{k}_{j} - \tilde{h}^{k}_{j}$ for all  $ k = 0, 1,2, \dots, N$ and $ j = 0, 1,2, \dots, M$.

So now using Equation~\eqref{discrete_full_version} we have
\begin{align}\label{discrete_full_version1}
&\sum_{s = 1}^{k}\big(\delta^{s+1}_{j}-\delta^{s}_{j} \big) E_{k,s,\tau} - \frac{1}{2}G_{\alpha} \vartheta(r_{j}) \bigg[ \big( \delta^{k+1}_{j+1}-\delta^{k+1}_{j-1} \big) + \big(  \delta^{k}_{j+1} - \delta^{k}_{j-1} \big) \bigg] \nonumber  \\
&-\frac{1}{2} \sum_{l = 1}^{j}\bigg[ \big(\delta^{k+1}_{l+1} - 2\delta^{k+1}_{l} + \delta^{k+1}_{l-1}\big) + \big(\delta^{k}_{l+1} - 2 \delta^{k}_{l} + \delta^{k}_{l-1}\big) \bigg] F_{k,j,l} = 0.
\end{align}
We define the grid point function as
\[
\delta^{k}(r) = 
\begin{cases}
\delta^{k}_{j}, \qquad \text{for} \qquad r_{j-\frac{h}{2}} < r < r_{j+\frac{h}{2}}, \\
0, \qquad \text{for} \qquad L - \frac{h}{2} < r \leq L,
\end{cases}
\]
endowed with the norm
\[
\Vert \delta^{k}(r) \Vert_{2} = \bigg( \sum_{j = 1}^{M-1} \xi \vert \delta^{k}_{j} \vert^{2}    \bigg)^{1/2} = \bigg( \int_{0}^{L} \vert \delta^{k}_{j} \vert^{2}  dr  \bigg)^{1/2}.
\]
Then we can expand $\delta^{k}(r)$ in a form of Fourier Series as
\[
\delta^{k}(r) = \sum_{a = -\infty}^{\infty} d_{k}(a) e^{\frac{2\pi ar}{L}}, \qquad  k = 1, 2, \dots, N,
\]
where 
\[
d_{k}(a)  = \frac{1}{L}\int_{0}^{L}\delta^{k}(r)  e^{-\frac{2\pi ar}{L}}.
\]

By applying the Parseval equality we have
\begin{equation}\label{Parseval_equality}
\int_{0}^{L} \vert \delta^{k}(r) \vert^{2}dr = \sum_{a = -\infty}^{\infty} \vert d_{k}(a) \vert^{2}.
\end{equation}
Hence we can write
\begin{equation}\label{Parseval_equality_extension}
\norm{\delta^{k} }_{2}^{2} = \sum_{a = -\infty}^{\infty} \vert d_{k}(a) \vert.
\end{equation}
This allows us to write the solution of Equation \eqref{discrete_full_version1} in the form
\[
\delta^{k}_{j} = d_{k}e^{ij\xi T}, \qquad \text{with} \quad T = \frac{2\pi a}{L}.
\]
By replacing $\delta^{k}_{j}$ into Equation \eqref{discrete_full_version1}, yields
\begin{align}
& -\bigg[iG_{\alpha}\vartheta(r_{j})e^{ij\xi T} \bigg(\frac{e^{i\xi T} - e^{-i\xi T}}{2i} \bigg) 
+ \bigg( \sum_{l=1}^{j} e^{il\xi T}\bigg(\frac{e^{i \xi T} + e^{-i\xi T}}{2}\bigg) F_{k,j,l} -  \sum_{l = 1}^{j} e^{il\xi T} F_{k,j,l} \bigg) \bigg]d_{k+1} \nonumber \\
& + \bigg[ - iG_{\alpha}\vartheta(r_{j})e^{ij\xi T} \bigg(\frac{e^{i\xi T} - e^{-i\xi T}}{2i} \bigg) + \bigg( \sum_{l=1}^{j} e^{il\xi T}\bigg(\frac{e^{i \xi T} + e^{-i\xi T}}{2}\bigg) F_{k,j,l} -  \sum_{l = 1}^{j} e^{il\xi T} F_{k,j,l} \bigg) \bigg]d_{k}\nonumber \\
& + e^{ij\xi T} \sum_{s = 1}^{k} \big( d_{s+1} - d_{s}\big)E_{k,s,\tau}  = 0. \nonumber 
\end{align}
This can be written in the form
\begin{align}\label{final_discrete1}
&d_{k+1}  \bigg[iG_{\alpha}\vartheta(r_{j})e^{ij\xi T} \sin(\xi T) 
+ \bigg( \sum_{l=1}^{j} e^{il\xi T} \cos(\xi T) F_{k,j,l} -  \sum_{l = 1}^{j} e^{il\xi T} F_{k,j,l} \bigg) \bigg]\nonumber \\
& = d_{k}\bigg[ - iG_{\alpha}\vartheta(r_{j})e^{ij\xi T} \sin(\xi T) + \bigg( \sum_{l=1}^{j} e^{il\xi T}\cos(\xi T) F_{k,j,l} -  \sum_{l = 1}^{j} e^{il\xi T} F_{k,j,l} \bigg) \bigg] \\
& - e^{ij\xi T} \sum_{s = 1}^{k} \big( d_{s+1} - d_{s}\big)E_{k,s,\tau}. \nonumber 
\end{align}
\begin{theorem}\label{Propos_stability}
The finite difference scheme given by Equation~\eqref{discrete_full_version} is stable.
\end{theorem}
\begin{proof}
To give a sketch of the proof of the Theorem \ref{Propos_stability}, we use mathematical induction. To do this, we consider Equation~\eqref{final_discrete} \medskip

For $k = 0$,
\begin{align}\label{final_discrete}
&d_{1}  \bigg[iG_{\alpha}\vartheta(r_{j})e^{ij\xi T} \sin(\xi T) 
+ \bigg( \sum_{l=1}^{j} e^{il\xi T} \cos(\xi T) F_{k,j,l} -  \sum_{l = 1}^{j} e^{il\xi T} F_{k,j,l} \bigg) \bigg]\nonumber \\
& = d_{0}\bigg[ - iG_{\alpha}\vartheta(r_{j})e^{ij\xi T} \sin(\xi T) + \bigg( \sum_{l=1}^{j} e^{il\xi T}\cos(\xi T) F_{k,j,l} -  \sum_{l = 1}^{j} e^{il\xi T} F_{k,j,l} \bigg) \bigg] \\
& - e^{ij\xi T} \sum_{s = 1}^{k} \big( d_{s+1} - d_{s}\big)E_{k,s,\tau}  = 0, \nonumber 
\end{align}
then
\begin{align*}
\vert d_{1} \vert &= \abs*{ \frac{\bigg( - iG_{\alpha}\vartheta(r_{j})e^{ij\xi T} \sin(\xi T) + \bigg( \sum_{l=1}^{j} e^{il\xi T}\cos(\xi T) F_{k,j,l} -  \sum_{l = 1}^{j} e^{il\xi T} F_{k,j,l} \bigg)\bigg)}{\bigg( iG_{\alpha}\vartheta(r_{j})e^{ij\xi T} \sin(\xi T) 
+ \bigg( \sum_{l=1}^{j} e^{il\xi T} \cos(\xi T) F_{k,j,l} -  \sum_{l = 1}^{j} e^{il\xi T} F_{k,j,l} \bigg) \bigg)}} \vert d_{0} \vert \\
&\leq \vert d_{0} \vert.
\end{align*}
Hence $\vert d_{1} \vert \leq \vert d_{0} \vert$ and from the Parseval equality~\eqref{Parseval_equality} and Equation~\eqref{Parseval_equality_extension} we have 
\[
\norm{\delta^{1}}_{2} \leq \norm{\delta^{0}}_{2}.
\]
We now assume that $\vert d_{m} \vert \leq \vert d_{0} \vert$, $m = 1,2,\dots,k$. We shall prove this for all values $m = k+1$.\medskip

Let $\abs*{\frac{\chi}{\tilde{\chi}}}\vert d_{0} \vert = \vert d_{1} \vert$, then 
\begin{align*}
\vert d_{k+1} \vert & \leq \abs*{\frac{\chi}{\tilde{\chi}}}\vert d_{0} \vert + \abs*{\frac{ e^{ij\xi T} \sum_{s = 1}^{k} \big( d_{s+1} - d_{s}\big)E_{k,s,\tau}}{\tilde{\chi}}} \\
\vert d_{k+1} \vert &\leq - \frac{B(\alpha)}{\alpha \tau} \big(1 - e^{-\lambda \tau} \big) \vert d_{k+1} \vert + \abs*{\frac{\chi}{\tilde{\chi}}}\vert d_{0} \vert  - \frac{B(\alpha)}{\alpha \tau} \big(1 - e^{-\lambda \tau} \big) \vert d_{k} \vert \\
 \vert d_{k+1} \vert  & \leq \abs*{\frac{1 - \frac{B(\alpha)}{\alpha \tau} \big(1 - e^{-\lambda \tau} \big) }{1 + \frac{B(\alpha)}{\alpha \tau} \big(1 - e^{-\lambda \tau} \big) }} \vert d_{0} \vert \\
 & \leq \vert d_{0} \vert.
\end{align*}
So from the Parseval equality~\eqref{Parseval_equality} and Equation~\eqref{Parseval_equality_extension} we have 
\[
\norm{\delta^{k}}_{2} \leq \norm{\delta^{0}}_{2}, \qquad k = 0,1,2,\dots, N.
\]
Thus the finite difference scheme given by Equation~\eqref{discrete_full_version} is stable.
\end{proof}

\section{Conclusion}\label{Sec:5}
We have presented a framework for existence and uniqueness of solutions of the problem of groundwater flow in a confined aquifer, within the concept of derivative with fractional order. We used the Fixed-Point Theorem as method to prove the existence, of the exact solution. We also succeed to show that, the modified groundwater flow equation obtained by changing the local time and space derivative by the Caputo-Fabrizio derivative, has an unique solution. The Crank-Nicolson scheme has been used to discretize our problem. In the light of this result, in order to show the efficiency of this new model, numerical analysis of stability of the solution and numerical simulations were presented for different values of $\alpha$. 


\section*{Acknowledgments}
The first author is indebted to the AIMS-Cameroon 2015--2016  tutor fellowship.\medskip


\section*{Author Contributions}

Each of the authors, J.D. Djida, A. Atangana, contributed to each part of this study equally and read and approved the final version of the manuscript.


\section*{Conflicts of Interest}

The authors declare no conflict of interest.


\begin{thebibliography}{10}
\bibitem{Cloot}
A.~Cloot and J.~F. Botha.
\newblock A generalised groundwater flow equation using the concept of non-integer order derivatives.
\newblock {\em Water SA}, vol.32, no.1, pp. 1--7, (2006).

\bibitem{Atangana2015}
Atangana~Abdon, and Alkahtani~Badr~Saad~T.
\newblock New model of groundwater flowing within a confine aquifer: application of Caputo-Fabrizio derivative.
\newblock{\em Arabian Journal of Geosciences}, vol.9,no.1,~pp.1--6, (2015).

\bibitem{Fabrizio}
Caputo,~M., Fabrizio~M.
\newblock A New Definition of Fractional Derivative Without Singular Kernel.
\newblock{\em Progress in Fractional Differentiation and Applications}, vol.1, no.2, pp. 73--85, (2015).

\bibitem{Losada}
Losada,~J., Nieto,~J.J.
\newblock Properties of a new fractional derivative without singular Kernel.
\newblock{\em Progress in Fractional Differentiation and Applications},vol.1, no.2, pp. 87--92, (2015)

\bibitem{Ivan2015}
I.~Area, J.~D.~Djida, J.~Losada, and Juan~J.~Nieto.
\newblock On Fractional Orthonormal Polynomials of a Discrete Variable.
\newblock{\em Discrete Dynamics in Nature and Society}, vol.~2015, Article ID 141325, 7 pages, 2015. doi:10.1155/2015/141325  

\bibitem{Abdon2}
Atangana~A.
\newblock On the new fractional derivative and application to nonlinear Fisher's reaction-diffusion equation,
\newblock{\em Applied Mathematics and Computation}, vol. 1, no.~273, pp. 948--956, (2016).

\bibitem{ciarlet2013}
Ciarlet,~P.~G.
\newblock {\em {Linear and Nonlinear Functional Analysis with Applications}}.
\newblock Society for Industrial and Applied Mathematics (SIAM, 3600 Market Street, Floor 6, Philadelphia, PA 19104), 2013.

\bibitem{Arman}
ARMAN~AGHILI and ALIREZA~ANSARI.
\newblock Solving partial fractional differential equations using the
FA-transform
\newblock {\em Arab J.Math Sci}, vol.19, no.1, pp. 61--71, (2013).

\bibitem{Abdon_Nieto}
Abdon~Atangana, and Juan~Jose~Nieto, 
\newblock Numerical solution for the model of RLC circuit via the fractional derivative without singular kernel,
\newblock{\em Advances in Mechanical Engineering}, vol.7, no.10, 2015.

\bibitem{Karatay}
Ibrahim~Karatay, Nurdane~ Kale,
and  Serife~R.~Bayramoglu,
\newblock A new difference scheme for time fractional heat equations based on the Crank-Nicholson method,
\newblock{\em Fractional Calculus and Applied Analysis}, vol.16, no.4, pp.892--910, 2013.


\end{thebibliography}

\end{document}